\documentclass[12pt,reqno,a4paper]{amsart}
\usepackage{amssymb,amsfonts,latexsym,mathrsfs,tikz}

\title[On paired roots systems of Coxeter groups]{On paired root systems of Coxeter groups}

\author{Fu, Xiang}

\dedicatory{\upshape
School of Mathematics and Statistics\\
University of Sydney, NSW 2006, Australia\\[.5em]
\texttt{xifu9119@mail.usyd.edu.au}\\
\texttt{xiangf@maths.usyd.edu.au}\\[1em]
Preliminary version,
\today
}

\newtheorem{theorem}{Theorem}[section]
\newtheorem{lemma}[theorem]{Lemma}
\newtheorem{proposition}[theorem]{Proposition}

\theoremstyle{definition}
\newtheorem{definition}[theorem]{Definition}
\newtheorem{remark}[theorem]{Remark}

\newtheorem{notation}[theorem]{Notation}
\newtheorem*{notn}{Notation}
\numberwithin{equation}{section}

\newcommand{\N}{\mathbb{N}}
\newcommand{\R}{\mathbb{R}}

\DeclareMathOperator{\PLC}{PLC}

\DeclareMathOperator{\spa}{span}
\DeclareMathOperator{\GL}{GL}

\DeclareMathOperator{\ord}{ord}


\begin{document}

\begin{abstract}
This paper examines a systematic method to construct a pair of (inter-related)
root systems for arbitrary Coxeter groups from a class of non-standard geometric 
representations. This method can be employed to construct generalizations
of root systems for a large family of linear groups generated by involutions. We then give a characterization
of Coxeter groups, among these groups, in terms of such paired 
root systems. Furthermore, we use this method to construct and study the paired root
systems for reflection subgroups of Coxeter groups.

\end{abstract}

\maketitle

\section{Introduction}

A \emph{Coxeter group} $W$ is an abstract group generated by a set of 
involutions $R$, called its \emph{Coxeter generators}, subject only to 
certain braid relations. Despite the simplicity of this definition, there
is a rich theory for Coxeter groups with non-trivial applications in a 
multitude of areas of mathematics and physics. When studying Coxeter groups, 
one of the most powerful tools we have at our disposal is the notion of 
\emph{root systems}. In classical literature (\cite[Ch.V, \S 4]{BN68} or \cite[\S 5.3--5.4]{HM}, for
example), the root system of a Coxeter group $W$ is a geometric construction 
arising from the \emph{Tits representation} of $W$. The Tits representation 
of $W$ is an embedding of $W$ into the orthogonal group of a certain bilinear 
form on a suitably chosen vector space $V$ subject to the requirement that
the $W$-conjugates of elements of $R$ are mapped to reflections with respect to 
certain hyperplanes in $V$. In the case that $W$ is finite, $V$ is Euclidean (of dimension 
equal to the cardinality of $R$), and 
the root system of $W$ simply consists of representative 
normal vectors for these hyperplanes. Those elements of the root system 
corresponding to the elements of $R$ are known as \emph{simple roots}, and in most
classical literature (\cite[Ch.V, \S 4]{BN68} and \cite[\S 5.3]{HM}, for example), the simple roots are
linearly independent. 

Similar constructions
of root systems can be extended to infinite Coxeter groups and Kac-Moody Lie algebras. However, the 
actual constructions of root systems differ depending whether the root systems are 
associated to Kac-Moody Lie algebras or infinite Coxeter groups. As discussed in 
the introduction of \cite{CH11}, while all definitions of root systems are related
to a given bilinear form, the actual bilinear forms considered in the case of 
Kac-Moody Lie algebras are different from the ones in Coxeter groups. 
Furthermore, it is well known (\cite{VD82} and \cite[Ch.3]{MD87})
that within an arbitrary Coxeter group $W$, all of its reflection subgroups are
themselves Coxeter groups, but in the literature (\cite{MD87} 
or \cite{MD90}, for example), the construction of the root systems corresponding to such 
reflections subgroups as subsets of the root system of $W$ requires special care. In particular,
since a proper reflection subgroup may have strictly more Coxeter generators than the over-group 
(as seen in \cite[Example 5.1]{CH11}), 
the equivalent of the simple roots in these root systems need not be linearly independent, 
making the overall theory of root systems and root bases somewhat less uniform.
Consequently, it seems profitable to 
develop a universal method for constructing root systems that is applicable to 
arbitrary Coxeter groups and their reflection subgroups, as well as 
to objects like Kac-Moody Lie algebras (in fact, to all groups with a  so-called 
\emph{root group datum}, as surveyed in \cite{PR09}).

In \cite{MP89}, \cite{BC01}, \cite{BC06} and \cite{PR09}, a number of more general 
notions of root systems have been studied. In these approaches, 
a pair of root systems are constructed in two vector spaces which are essentially algebraic duals
of each other (apart from \cite{MP89}, the two vector spaces involved are explicitly required to
be algebraic duals of each other, whereas in \cite{MP89} the two vector spaces are linked by a non-degenerate
bilinear pairing satisfying certain integrality conditions).   

Recently, an approach
taken in \cite{FU0} and \cite{FU2} generalizes those of \cite{MP89}, \cite{BC01}, \cite{BC06} and \cite{PR09}. 
In this approach, for an arbitrary Coxeter group, a pair of root systems are
constructed in two vector spaces linked only by a bilinear pairing which does not require 
the non-degeneracy and integrality conditions of \cite{MP89}. This particular approach
allows more abstract geometry in the two root systems to take place (for example, the two
representation spaces need not be algebraic duals of each other), whilst 
providing a unified theory of root systems, especially with respect to reflection subgroups
(this last point is to be established in Section~\ref{sec:reflection} of this present paper).   

In this paper, we present a few results demonstrating the 
``universalness'' of  the notion of root systems in \cite{FU0} and \cite{FU2}. 
In fact, this new approach applies to a large family of linear groups  
generated by involutions, and one of the key results of this paper (Theorem \ref{th:key}) 
shows that these groups are Coxeter groups only if the corresponding root systems 
decompose as disjoint unions of those roots generalizing the classical
concept of \emph{positive roots} and those roots generalizing the classical concept of
\emph{negative roots}. In fact, this result provides an alternative characterization for 
Coxeter groups, since it is well known that for any Coxeter group we may construct 
a root system that decomposes in the same way. This alternative characterization is 
implicitly suggested in the work of Matthew Dyer (\cite{MD87}), and we are very grateful 
to Prof.~Dyer for a large number of helpful suggestions leading to the development of 
this generalized notion of root systems.

The main body of this paper is organized into 2 sections, namely, 
Section~\ref{sec:prr} and Section~\ref{sec:reflection}. In Section~\ref{sec:prr}
we develop a notion of root systems applicable to a large family of groups that are generated 
only by involutions, and we investigate when these root systems may decompose into
disjoint unions of the so-called \emph{positive roots} and the so-called
\emph{negative roots}, and we prove that such groups are Coxeter groups only if
such decompositions take place (Theorem~\ref{th:datum} and Theorem~\ref{th:key}). 
In Section~\ref{sec:reflection} we prove that the 
notion of root systems in \cite{FU0} and \cite{FU2} applies to all the reflection 
subgroups of any Coxeter group. In particular, we give a geometric characterization
of the roots that correspond to the Coxeter generators of reflection subgroups (Proposition~\ref{pp:can1}
and Proposition~\ref{pp:d3.4}), and we show that these characterizations are precisely those
allowing these roots to be \emph{the simple roots} for the root systems of such reflection subgroups
within the root systems of the respective over-groups.  
\begin{notn}
If $A$ is a subset of a real vector space then we define the
\emph{positive linear cone} of $A$, denoted $\PLC(A)$, to be the set
$$
\{\, \sum_{a\in A}c_a a\mid 
              \text{$c_a \geq 0$ for all $a\in A$, and $c_{a'}>0$ for some $a'\in A$ } \,\}.
$$ 
Furthermore, we define $-A:=\{\,-v\mid v\in A\,\}$. Also, if $B$ is a subset of 
a group $G$ then $\langle B \rangle$ denotes the subgroup of $G$ generated by $B$.
\end{notn}

\section{Decomposition of Root Systems and Coxeter Datum}
\label{sec:prr}

Let $V_1$ and $V_2$ be vector spaces over the real field $\R$ equipped with a bilinear
pairing $\langle\,,\,\rangle\colon V_1\times V_2 \to \R$. Let $S$ be an indexing set, 
and suppose that
 $\Pi_1 :=\{\, \alpha_s\mid s\in S\, \}\subseteq V_1$ and
 $\Pi_2 := \{\, \beta_s\mid s\in S\, \}\subseteq V_2$
are both in bijective correspondence with $S$. Further, suppose that $\Pi_1$ and 
$\Pi_2$ satisfy the following conditions:
\begin{itemize}
 \item [(D1)] $\langle\alpha_s, \beta_s\rangle =1$ for all $s\in S$;
 \item [(D2)] 
 \begin{itemize}
 \item[(i)] $0\notin\PLC(\Pi_1)$ and $0\notin\PLC(\Pi_2)$.
 \item[(ii)] $\alpha_s\notin\PLC(\Pi_1\setminus\{\alpha_s\})$ and 
$\beta_s\notin\PLC(\Pi_2\setminus\{\beta_s\})$ for each $s\in S$.
\end{itemize}
\end{itemize}
Observe that condition (D2) (i) implies that $\alpha_s\notin\PLC(-\Pi_1\setminus\{-\alpha_s\})$ and 
$\beta_s\notin\PLC(-\Pi_2\setminus\{-\beta_s\})$ for each $s\in S$.
We remark that there do exist examples for which $\Pi_1$
(resp.~$\Pi_2$) is linearly dependent, in which case necessarily
some $\alpha_s$ (resp.~$\beta_s$) will be expressible as a linear
combination of $\Pi_1\setminus\{\alpha_s\}$ (resp.
 $\Pi_2\setminus\{\beta_s\}$) with coefficients of mixed signs.
\begin{definition}
\label{def:start}
For $s\in S$, define $\rho_1(s)\in \GL(V_1)$ and $\rho_2(s)\in \GL(V_2)$ 
by the rules
\begin{align*}
 \rho_1(s)(x)&:=x-2\langle x, \beta_s\rangle \alpha_s\\
\noalign{\hbox{for all $x\in V_1$, and }}
 \rho_2(s)(y)&:=y-2\langle \alpha_s, y \rangle \beta_s
\end{align*}
for all $y\in V_2$. Further, we define, for each $i\in \{\,1, 2\,\}$: 
\begin{align*}
 R_i &:=\{\,\rho_i(s)\mid s\in S\,\};\\
 W_i &:=\langle R_i\rangle; \\
 \Phi_i &:= W_i \Pi_i;\\
 \Phi^+_i & := \Phi_i\cap \PLC(\Pi_i);\\
\noalign{\hbox{and}}
 \Phi^-_i & :=-\Phi^+_i.
\end{align*}
For each $i\in \{1, 2\}$, and for each $s\in S$, we call $\rho_i(s)$ the 
\emph{reflections} corresponding to $s$ in $V_i$.  
We call $\Phi_i$ the \emph{root system} for the \emph{Weyl group} $W_i$ 
realized in $V_i$, and we call $\Pi_i$ the set of \emph{simple roots} in $\Phi_i$. 
Furthermore, we call $\Phi^+_i$
the set of \emph{positive roots} in $\Phi_i$ and $\Phi^-_i$ the set of 
\emph{negative roots} in $\Phi_i$.
\end{definition}

\begin{remark}
For each $i\in \{1, 2\}$ and each $s\in S$ note that  $\rho_i(s)$ is an involution with a
$-1$-eigenvector of multiplicity $1$.
Furthermore, it is a consequence of condition (D2) that 
$\Phi_i^+\cap \Phi_i^-=\emptyset$. Use $\uplus$ to denote disjoint unions, we have: 
\end{remark}

\begin{theorem}
 \label{th:datum}
Given conditions (D1) and (D2), the following are equivalent\textup{:}
\begin{itemize}
\item[(i)] $\Phi_1=\Phi^+_1\uplus \Phi^-_1$.

\item[(ii)] $\Phi_2=\Phi^+_2\uplus \Phi^-_2$.

\item[(iii)] For all $s, t\in S$ the following three conditions are satisfied\textup{:}
\begin{itemize}
 \item [\quad(D3)] $\langle \alpha_s, \beta_t\rangle\leq 0$ and $\langle \alpha_t, \beta_s\rangle \leq 0$
              whenever $s\neq t$.
 \item [\quad(D4)] $\langle \alpha_s, \beta_t\rangle =0$ if and only if $\langle \alpha_t, \beta_s\rangle=0$. 
 \item [\quad(D5)] Either $\langle \alpha_s, \beta_t\rangle \langle \alpha_t, \beta_s\rangle =\cos^2\frac{\pi}{m_{st}}$
              for some integer $m_{st}\geq 2$, or else
              $\langle \alpha_s, \beta_t\rangle \langle \alpha_t, \beta_s\rangle\geq 1$. 
\end{itemize}
\end{itemize}
\end{theorem}

It is a consequence of this theorem that if any of the equivalent conditions in it is satisfied
then $W_1$ and $W_2$ are isomorphic Coxeter groups.
To prove this theorem we shall need a few technical results first. These are essentially taken from 
\cite{MD87}, and for completeness, the relevant proofs are included here.

Let $\mathscr{A}$ be a commutative $\R$-algebra, let $q^{1/2}$ and $X$ be units of $\mathscr{A}$, and let 
$\gamma \in \R$. Define $A$, $B$ to be $2\times 2$ matrices over $\mathscr{A}$ given by
\begin{equation*}
 A =\begin{pmatrix}
-1 & 2 \gamma q^{1/2}X\\
0 & q 
    \end{pmatrix}
\qquad\qquad
B=\begin{pmatrix}
q & 0 \\
2\gamma q^{1/2}X^{-1} & -1   
  \end{pmatrix}.
\end{equation*}
It is easily proved by induction on $n\in \N$ that 
\begin{align}
\label{eq:MA1}
 B(AB)^n &=
\begin{pmatrix}
q^{n+1}p_{2n+1} & -q^{n+\tfrac{1}{2}}p_{2n}X \\
q^{n+\tfrac{1}{2}}p_{2n+2}X^{-1} & -q^{n}p_{2n+1} 
\end{pmatrix}\\
\label{eq:MA1a}
A(BA)^n &=
\begin{pmatrix}
-q^{n}p_{2n+1} & q^{n+\tfrac{1}{2}}p_{2n+2}X \\
-q^{n+\tfrac{1}{2}}p_{2n}X^{-1} & q^{n+1}p_{2n+1} 
\end{pmatrix}\\
\label{eq:MA2a}
(BA)^n & =
\begin{pmatrix}
 -q^n p_{2n-1} & q^{n+\tfrac{1}{2}}p_{2n}X \\
-q^{n-\tfrac{1}{2}}p_{2n}X^{-1} & q^n p_{2n+1}
\end{pmatrix}\\
\noalign{\hbox{and}}
\label{eq:MA2}
(AB)^n & =
\begin{pmatrix}
 q^n p_{2n+1} & -q^{n-\tfrac{1}{2}}p_{2n}X \\
q^{n+\tfrac{1}{2}}p_{2n}X^{-1} & -q^n p_{2n-1}
\end{pmatrix}
\end{align}
where $p_n \in \R$ ($n \in \{-1\} \cup \N$) are defined recursively by 
\begin{equation}
 \label{eq:MA3}
p_{-1} = -1, \qquad p_0 =0, \qquad p_{n+1} = 2\gamma p_n -p_{n-1} \quad \text{($n\in \N$)}.
\end{equation}
The solutions of the recurrence equation (\ref{eq:MA3}) is 
\begin{equation}
 \label{eq:MA4}
p_n =
\begin{cases}
 n &\gamma =1 \\
(-1)^{n+1}n&\gamma =-1 \\[5pt]
\dfrac{\sinh n\theta}{\sinh \theta}\quad\text{(where $\theta=\cosh^{_1} \gamma$)}&|\gamma|>1\\[10pt]
\dfrac{\sin n\theta}{\sin\theta} \quad \text{(where $\theta = \cos^{-1}\gamma$)} &|\gamma|<1.
\end{cases}
\end{equation} 

\begin{proposition}\textup{(\cite[Lemma 2.2]{MD87})}
 \label{pp:pn}
Keep all the above notation.

\noindent(i)\quad Conditions (1) and (2) below are equivalent:
\begin{itemize}
 \item [(1)] $p_n p_{n+1}\geq 0$ for all $n\in \N$; 
 \item [(2)] $\gamma\in \{\,\cos\frac{\pi}{m}\mid m\in \N, m\geq 2\,\}\cup [1, \infty)$.
\end{itemize}

\noindent(ii)\quad If $\gamma=\cos\frac{k\pi}{m}$ for some $k, m\in \N$ with 
              $0<k<m$ then the matrices $A$ and $B$ satisfies the equation
              $$ABA\cdots =BAB\cdots$$
              where there are $m$ factors on each side.

\noindent(iii)\quad If $q=1$ then the matrix $AB$ has order $m$ if $\gamma=\cos\frac{k\pi}{m}$
                    for some $k, m\in \N$ with $0<k<m$ and $\gcd(m,k)=1$, and the matrix $AB$ has infinite 
                    order otherwise. 
\end{proposition}
\begin{proof}
(i):\quad First assume that (1) holds. Observe that (\ref{eq:MA3}) yields that $p_1=1$ and $p_2=2 \gamma$, 
hence $\gamma\geq 0$. Since (2) obviously holds if $\gamma\geq 1$, we may assume that 
$0\leq \gamma<1$. Choose $\theta$ so that $0<\theta\leq \frac{\pi}{2}$ and $\cos\theta=\gamma$, and 
let $m$ be the largest integer such that 
$$0<\theta< 2\theta<\cdots< m\theta\leq \pi.$$
Note that $m\geq 2$. Now if $m\theta\neq \pi$ then $\pi<(m+1)\theta<2\pi$, and in view of (\ref{eq:MA4})
we have $p_m=\frac{\sin m\theta}{\sin \theta}>0$, whereas $p_{m+1}=\frac{\sin(m+1)\theta}{\sin\theta}<0$, 
contradicting (1). Hence $m\theta=\pi$ and $\gamma=\cos\frac{\pi}{m}$ for some integer $m\geq 2$, whence (2) 
holds as desired. Conversely, if (2) holds then it follows immediately from (\ref{eq:MA4}) that (1) holds.

\noindent(ii):\quad If $m=2r$ is even then our task is to prove that $(AB)^r=(BA)^r$. It follows from 
(\ref{eq:MA4}) that $p_n=\frac{\sin(n k \pi/{2r})}{\sin(k \pi/{2r})}$, and hence, $p_{2r+1}=(-1)^k$
and $p_{2r-1}=(-1)^{k+1}$, while $p_{2r}=0$. Then it follows from (\ref{eq:MA2a}) and (\ref{eq:MA2}) that
$(AB)^r=(BA)^r$.

If $m=2r+1$ is odd then our task is to prove that $B(AB)^r=A(BA)^r$. In this case we find from (\ref{eq:MA4})
that $p_{2r+1}=0$, while $p_{2r+2}=(-1)^k$ and $p_{2r}=(-1)^{k+1}$, and then the required result follows 
immediately from (\ref{eq:MA1a}) and (\ref{eq:MA1}).

\noindent(iii):\quad If $\gamma=\cos\frac{k\pi}{m}$ then it follows immediately from (ii) above that 
$(AB)^m=1$, because $A^2=B^2=1$ when $q=1$. If $0<n<m$ and $\gcd(k,m)=1$, 
then (\ref{eq:MA4}) yields that
$p_n=\frac{\sin(nk\pi/m)}{\sin(k\pi/m)}\neq 0$, and it then follows from (\ref{eq:MA2}) that
$(AB)^n\neq 1$, proving that $AB$ has order $m$. On the other hand, if $\gamma$ 
is of any other form then it follows 
from (\ref{eq:MA4}) that $p_n\neq 0$ for all integer $n>0$. Then it is clear from (\ref{eq:MA2})  that 
$(AB)^n\neq 1$ for all such $n$, proving that $AB$ has infinite order.
\end{proof}
 
Now we are ready to prove Theorem \ref{th:datum}.

\begin{proof}[Proof of Theorem \ref{th:datum}]
We give a proof that (i) is equivalent to (iii). An entirely similar argument shows that (ii) and
(iii) are also equivalent.

First we show that (iii) implies (i). Given conditions (D3), (D4) and (D5) of the present paper we observe that 
$\mathscr{C}:=(\, S, V_1, V_2, \Pi_1, \Pi_2, \langle, \rangle\,)$ 
forms a Coxeter datum in the sense of \cite{FU2}, and hence (i) follows immediately from 
Lemma 3.2 of \cite{FU2}.

Conversely, suppose that $\Phi_1=\Phi^+_1\uplus \Phi^-_1$. We first prove that (D3) holds. Let $s, t\in S$ be distinct. 
By definition we have 
\begin{equation}
 \label{eq:c1}
\rho_1(t)\alpha_s=\alpha_s-2\langle \alpha_s, \beta_t\rangle \alpha_t.
\end{equation}
The condition 
$\Phi_1=\Phi^+_1\uplus \Phi^-_1$ implies that either 
\begin{equation}
 \label{eq:c2}
\rho_1(t)\alpha_s=\sum_{r\in S}c_r \alpha_r, \text{   where all $c_r\geq 0$}, 
\end{equation}
or else
\begin{equation}
 \label{eq:c3}
\rho_1(t)\alpha_s=\sum_{r\in S}-c_r \alpha_r, \text{   where all $c_r\geq 0$}.
\end{equation}
The following argument involving inspecting the coefficients rules out the possibility of (\ref{eq:c3}). 
Indeed, in view of (\ref{eq:c1}) we would have from (\ref{eq:c3}) that
$$(1+c_s)+\sum_{r\in S\setminus\{s, t\}}c_r \alpha_r=(2\langle \alpha_s, \beta_t\rangle -c_t)\alpha_t.$$
Now if $2\langle \alpha_s, \beta_t\rangle -c_t>0$ then we have a contradiction to (D2), 
since then $\alpha_t\in \PLC(\Pi_1\setminus\{\alpha_t\})$; whereas if 
$2\langle \alpha_s, \beta_t\rangle -c_t\leq 0$ then we again have a contradiction
to (D2), since then $0\in \PLC(\Pi_1)$. 
Thus (\ref{eq:c2}) must be the case, and in view of (\ref{eq:c1}) we have 
$$(1-c_s) \alpha_s = (2\langle \alpha_s, \beta_t\rangle +c_t)\alpha_t + \sum_{r\in S\setminus\{s, t\}}c_r \alpha_r.$$
Suppose  for a contradiction that $\langle \alpha_s, \beta_t\rangle >0$. 
Then $2\langle \alpha_s, \beta_t\rangle +c_t>0$.
Now if $1-c_s>0$ then we have a contradiction to condition (D2), since 
$\alpha_s\in \PLC(\Pi_1\setminus\{\alpha_s\})$; 
whereas if $1-c_s\leq 0$ then we again have a contradiction to (D2), since 
$0\in \PLC(\Pi_1)$.
Thus it follows from these contradictions that $\langle \alpha_s, \beta_t\rangle\leq 0$, 
and interchange the roles of $s$ and $t$, we see 
that $\langle \alpha_t, \beta_s\rangle \leq 0$, whence 
(D3) holds.  

Next, suppose that further $\langle \alpha_s, \beta_t\rangle =0$, and we prove that (D4) holds. Observe that
\begin{align*}
 \rho_1(t)\rho_1(s)\alpha_t =\rho_1(t)(\rho_1(s)\alpha_t) 
                           &=-\alpha_t-2\langle \alpha_t, \beta_s\rangle \alpha_s
                           +4\langle \alpha_t, \beta_s\rangle \langle\alpha_s,\beta_t\rangle \alpha_t\\
                          &=-\alpha_t-2\langle \alpha_t, \beta_s\rangle \alpha_s.
\end{align*}
Again the assumption that $\Phi_1=\Phi^+_1\uplus \Phi^-_1$ implies that either
\begin{equation}
 \label{eq:c4}
-\alpha_t-2\langle \alpha_t, \beta_s\rangle \alpha_s =\sum_{r\in S} c_r \alpha_r, \text{   where all $c_r\geq 0$,}
\end{equation}
or else
\begin{equation}
 \label{eq:c5}
-\alpha_t-2\langle \alpha_t, \beta_s\rangle \alpha_s =\sum_{r\in S} -c_r \alpha_r, \text{   where all $c_r\geq 0$.}
\end{equation}
A similar argument involving inspecting the coefficients together with (D2) make 
it possible to conclude that only (\ref{eq:c5}) is possible.
Hence 
\begin{equation}
\label{eq:c6}
(-2\langle \alpha_t, \beta_s\rangle +c_s)\alpha_s+\sum_{r\in S\setminus\{s, t\}}c_r \alpha_s =(1-c_t)\alpha_t.  
\end{equation}
Observe that (D3) just prove above yields that $-\langle \alpha_t, \beta_s\rangle \geq 0$.
Now if $1-c_t<0$ then we will have a contradiction to (D2), since then $0\in \PLC(\Pi_1)$; 
whereas if $1-c_t>0$ then we again have a contradiction to (D2), since then 
$\alpha_t\in \PLC(\Pi_1\setminus\{\alpha_s\})$. Thus $c_t =1$, and then  
(\ref{eq:c6}) implies, in view of (D2) and (D3), that $\langle \alpha_t, \beta_s\rangle =0=c_s$ 
(and $c_r=0$ for all $r\in S\setminus\{s, t\}$).
Interchange the roles of $s$ and $t$ we deduce that
$\langle \alpha_t, \beta_s\rangle =0$ implies that $\langle \alpha_s, \beta_t\rangle =0$, whence (D4) holds.

To prove that (D5) holds, we may assume that 
$\langle \alpha_s, \beta_t\rangle\langle \alpha_t, \beta_s\rangle\neq 0$, for otherwise
$\langle \alpha_s, \beta_t\rangle\langle \alpha_t, \beta_s\rangle=\cos^2\frac{\pi}{2}$, trivially 
satisfying (D5). We let $\mathscr{A}$, $\gamma$, $q$, $X$, $p_n$, $A$ and $B$ be as defined before 
Proposition \ref{pp:pn}. If we set
\begin{align*}
 \mathscr{A}&=\R;\\
           q&= 1; \\
           \gamma&=\sqrt{\langle \alpha_s, \beta_t\rangle\langle \alpha_t, \beta_s\rangle};\\
\noalign{\hbox{and}}
          X &=\frac{-\langle \alpha_t, \beta_s\rangle}{\sqrt{\langle \alpha_s, \beta_t\rangle\langle \alpha_t, \beta_s\rangle}},
 \end{align*}
then it is readily checked that $A$ and $B$ are the matrices representing the actions of $\rho_1(s)$ and $\rho_1(t)$ respectively, 
on the $\langle \{\, \rho_1(s), \rho_1(t)\,\}\rangle$-invariant subspace $\R\alpha_s+\R\alpha_t$. It follows from (\ref{eq:MA1}) to
(\ref{eq:MA2}) and a similar argument involving inspecting the coefficients as used above that the requirement 
$$ \langle \{\, \rho_1(s), \rho_1(t)\,\}\rangle \alpha_s\cup \langle \{\, \rho_1(s), \rho_1(t)\,\}\rangle\alpha_t
\subseteq \Phi^+_1\uplus \Phi^-_1 $$
is equivalent to $p_n p_{n+1}\geq 0$ for all $n\in \N$. By Proposition~\ref{pp:pn}, this later condition is, in turn, 
equivalent to 
$$\langle \alpha_s, \beta_t\rangle\langle \alpha_t, \beta_s\rangle\in \{\, \cos^2\frac{\pi}{m} \mid m\in \N_{\geq 2}\,\}\cup[1, \infty),$$
whence (D5) holds, finally establishing that (i) implies (iii).
\end{proof}

\begin{notation}
For $w_i\in W_i$ (for each $i\in \{1, 2\}$), let $\ord(w_i)$ denote the order of $w_i$ in $W_i$.
For those $s, t\in S$ with $\langle \alpha_s, \beta_t\rangle \langle \alpha_t, \beta_s\rangle \geq 1$, 
extend the definition of $m_{st}$ (given in Theorem~\ref{th:datum}) by setting  $m_{st}=\infty$.
\end{notation}

\begin{proposition}
 \label{pp:order}
Suppose that one of the (equivalent) statements of 
Theorem \ref{th:datum} is satisfied. Then 
$\ord(\rho_i(s)\rho_i(t))=m_{st}$.
\end{proposition}
\begin{proof}
If one of the (equivalent) statements of Theorem \ref{th:datum} is satisfied, then 
$\mathscr{C}:=(\, S, V_1, V_2, \Pi_1, \Pi_2, \langle, \rangle\,)$ 
forms a Coxeter datum in the sense of \cite{FU2}, and thus the required result 
follows immediately from Proposition~2.8 of \cite{FU2}.
\end{proof}

We point out that a Coxeter datum in the sense of \cite{FU2}
automatically satisfies the conditions (D1) to (D5) of the present paper. 
Indeed, the only possible difference of these two formulations is that in (D2) of 
the present paper we require a seemingly extra condition that 
$\alpha_s\notin\PLC(\Pi_1\setminus\{\alpha_s\})$ and 
$\beta_s\notin\PLC(\Pi_2\setminus\{\beta_s\})$ for each $s\in S$. 
However, it can be checked that this condition is an immediate 
consequence of (C1), (C2) and (C5) of a Coxeter datum in the sense of \cite{FU2}
(in fact, this is just \cite[Lemma 2.5]{FU2}). Thus we
have:

\begin{proposition}
\label{pp:eqv}
 The following are equivalent:
\begin{itemize}
 \item [(i)] $\mathscr{C}:=(\, S, V_1, V_2, \Pi_1, \Pi_2, \langle\,,\,\rangle\,)$ satisfies 
             one of the (equivalent) statements of Theorem \ref{th:datum} ;
 \item [(ii)] $\mathscr{C}:=(\, S, V_1, V_2, \Pi_1, \Pi_2, \langle\,,\,\rangle\,)$ is a 
              Coxeter datum in the sense of \cite{FU2}.
\end{itemize} 
\qed
\end{proposition}

Next we have a result which enables us to give a characterization of Coxeter groups, among a large 
family of linear groups that are generated by involutions, in terms of their root systems:
\begin{theorem}
\label{th:key}
Let $S$, $\Pi_1$ and $\Pi_2$ be the same as at the beginning of this section, and
let $R_1$, $W_1$, $\Phi_1$, $R_2$, $W_2$ and $\Phi_2$ be as in Definition~\ref{def:start}.
Let $(W, R)$ be a Coxeter system in the sense of \cite{BN68} or \cite{HM}, with $W$ being
an abstract group generated by a set of involutions $R:=\{\,r_s\mid s\in S \,\}$ subject
only to the condition that for $s, t\in S$ the order of $r_s r_t$ is either 
equal to $m$ if $\langle \alpha_s, \beta_t\rangle \langle \alpha_t, \beta_s\rangle =\cos^2(\pi/m)$, 
or else equal to infinity. 
Then $\Phi_1=\Phi^+_1\uplus\Phi^-_1$, or equivalently, $\Phi_2=\Phi^+_2\uplus\Phi^-_2$ only if
there exist isomorphisms $f_1\colon W\to W_1$ and $f_2\colon W\to W_2$ such that $f_1(r_s)=\rho_1(s)$
and $f_2(r_s)=\rho_2(s)$ for all $s\in S$. 
\end{theorem}
\begin{proof}
Follows immediately from Proposition \ref{pp:eqv} above and \cite[Theorem 2.10]{FU2}.
\end{proof}

\begin{remark}
Theorem \ref{th:key} shows that if $\Phi_1=\Phi^+_1\uplus\Phi^-_1$, 
or equivalently, $\Phi_2=\Phi^+_2\uplus\Phi^-_2$ then $(W_1, R_1)$
and $(W_2, R_2)$ are Coxeter systems isomorphic to $(W, R)$. It is well known 
in the literature that all Coxeter groups have root systems decomposable into a 
disjoint union of positive roots and negative roots ( 
\cite[Proposition 4.2.5]{ABFB} or \cite[\S 5.4]{HM}, for example). Furthermore, given 
an arbitrary Coxeter system $(W', R')$, it follows from \cite{FU0} and \cite{FU2}
that we could associate a Coxeter datum 
$\mathscr{C'}:=(\, S', V'_1, V'_2, \Pi'_1, \Pi'_2, \langle\,,\,\rangle'\,)$ to $(W', R')$, 
such that the paired root systems $\Phi'_1$ and $\Phi'_2$ arising from this particular 
Coxeter datum admit decompositions 
$\Phi'_1=\Phi'^+_1\uplus\Phi'^-_1$ and $\Phi'_2=\Phi'^+_2\uplus\Phi'^-_2$. 
These facts combined with Theorem \ref{th:key} yield that
if a linear group is generated by involutions, 
then it is a Coxeter group if and only if it has a root system decomposable 
into a disjoint union of positive roots and negative roots. 
\end{remark}

Let $W$ and $R$ be as in Theorem \ref{th:key}, we call $(W, R)$ the \emph{abstract Coxeter system}
corresponding to $\mathscr{C}$ with $W$ being the corresponding \emph{abstract Coxeter group}.
We see immediately from the above theorem that $f_1$ and $f_2$ give rise to faithful $W$-actions
on $V_1$ and $V_2$ in the natural way with $wx:=(f_1(w))(x)$ and $w y:=(f_2(w))(y)$ for all $w\in W$,
$x\in V_1$ and $y\in V_2$.

To close this section we include the following useful result taken from \cite{FU2}:
\begin{lemma}
 \label{lem:equivariant}
(i)\quad $\langle\,,\,\rangle$ is $W$-invariant, that is, 
$\langle wx, wy\rangle =\langle x, y\rangle$ for all $w\in W$, $x\in V_1$ and $y\in V_2$.

\noindent (ii)\quad There exists a $W$-equivariant bijection $\phi\colon \Phi_1\to \Phi_2$
                    satisfying $\phi(\alpha_s)=\beta_s$ for all $s\in S$.

\noindent (iii)\quad Let $\phi$ be as in (ii) above, and let $x, x'\in \Phi_1$. Then 
$\langle x, \phi(x')\rangle =0$ if and only if $\langle x', \phi(x)\rangle =0$.
\end{lemma}
\begin{proof}
 (i):\quad Lemma 2.13 of \cite{FU2}.

\noindent (ii):\quad Proposition 3.18 of \cite{FU2}.

\noindent (iii):\quad Corollary 3.25 of \cite{FU2}.
\end{proof}

For the rest of this paper, the notation $\phi$ will be fixed for the $W$-equivariant bijection 
in Lemma~\ref{lem:equivariant}~(iii).

\section{Reflection Subgroups and Canonical Generators in Coxeter Groups}
\label{sec:reflection}

Given a Coxeter group $W$ and its Coxeter generators $R$, a subgroup $W'$ of $W$
is called a \emph{reflection subgroup} if $W'$ is generated by those elements of the form $w r w^{-1}$ 
(where $w\in W$ and $r\in R$). It is well known that $W'$ is a Coxeter group, 
and consequently the notion of a Coxeter datum as in the previous section applies to $W'$. In this section
we study the paired root systems for $W'$ as a subsets of the paired root systems for $W$. In the spirit of the 
previous section, our investigation of the paired root systems for $W'$ is based on a Coxeter datum $\mathscr{C'}$
closely related to the Coxeter datum for 
the over group $W$. In particular, we show that the Coxeter generators of $W'$ 
are characterize by this Coxeter datum $\mathscr{C'}$. 
In addition to obtaining certain geometric insights of reflection subgroups of 
Coxeter groups, these investigations also establish the fact that the method of constructing paired
root systems via Coxeter data applies to paired root systems for reflection 
subgroups of a Coxeter group, either on their own
or as subsets of the paired root systems of the over group.

Suppose that 
$\mathscr{C}:=(\, S, V_1, V_2, \Pi_1, \Pi_2, \langle\,,\,\rangle\,)$ satisfies
conditions (D1) to (D5) of Section~\ref{sec:prr} inclusive (or in view of Proposition \ref{pp:eqv}, we
could equivalently suppose that $\mathscr{C}$ is a Coxeter datum in the sense of \cite{FU2}).
Let $(W, R)$ be the abstract Coxeter system associated to the Coxeter datum $\mathscr{C}$, and 
keep all the notation of the previous section.

Let $T:=\bigcup_{w\in W} w R w^{-1}$, and call it the \emph{set of reflections} in $W$.
For $s\in S$ and $w\in W$, observe that for each $x\in V_1$ and $y\in V_2$, 
Lemma~\ref{lem:equivariant} yields that
\begin{align}
 \label{eq:ref1}
 w r_s w^{-1} x=w(w^{-1}x-2\langle w^{-1}x, \beta_s\rangle \alpha_s)&=x-2\langle w^{-1}x, \beta_s\rangle w\alpha_s\notag\\
                                                                    &=x-2\langle x, \phi(w \alpha_s)\rangle w \alpha_s,
\end{align}
and 
\begin{align}
 \label{eq:ref2}
 w r_s w^{-1} y=w(w^{-1}y-2\langle \alpha_s, w^{-1}y\rangle \beta_s)&=y-2\langle w\alpha_s, y\rangle w\beta_s\notag\\
                                                                    &=y-2\langle \phi^{-1}(w\beta_s), y\rangle w \beta_s.
\end{align}
Now suppose that $\alpha\in \Phi_1$ and $\beta\in \Phi_2$ are arbitrary. 
Then $\alpha=w_1 \alpha_s$ and $\beta=w_2\beta_t$ for some $w_1, w_2\in W$ and $s, t\in S$. 
It follows from (\ref{eq:ref1}) and (\ref{eq:ref2}) that we can unambiguously define 
$r_{\alpha}, r_{\beta}\in T$, the \emph{reflection corresponding to} $\alpha$ and $\beta$
respectively, by 
\begin{align}
\label{eq:ref3}
 r_{\alpha}&= r_{w_1\alpha_s}:=w_1 r_s w_1^{-1},\\
\noalign{\hbox{and}}
\label{eq:ref4}
r_{\beta}&=r_{w_2 \beta_t}:= w_2 r_t w_2^{-1},
\end{align}
with
\begin{align*}
 r_{\alpha} x&= x-2 \langle x, \phi(\alpha)\rangle \alpha\\
\noalign{\hbox{for all $x\in V_1$ and}}
 r_{\beta} y&= y-2\langle \phi^{-1}(\beta), y\rangle \beta
\end{align*}
for all $y\in V_2$.

\begin{definition}
(i)\quad A subgroup $W'$ of $W$ is called a \emph{reflection subgroup}
if $W'=\langle W'\cap T\rangle$.

\noindent{(ii)}\quad For each $i\in \{1, 2\}$, a subset $\Phi'_i$ of $\Phi_i$
is called a \emph{root subsystem} if $r_x y\in \Phi'_i$ whenever $x, y\in \Phi'_i$.

\noindent{(iii)}\quad If $W'$ is a reflection subgroup, set
$\Phi_i(W'):=\{\,x\in \Phi_i\mid r_x\in W'\,\}$ for each $i\in {1, 2}$.
\end{definition}

\begin{lemma}
\label{lem:stab}
Let $W'$ be a reflection subgroup of $W$. Then for each $i\in\{1,2\}$
$$
W'\Phi_i(W')=\Phi_i(W').
$$
\end{lemma}
\begin{proof}
We prove that $W'\Phi_1(W')=\Phi_1(W')$ here, and we stress that the 
other half follows in the same way. Let $w\in W'$. By definition, 
we have $w=t_1t_2\cdots t_n$ where $t_1, t_2, \ldots, t_n \in W'\cap T$. 
Now let $x\in \Phi_1(W')$ 
be arbitrary.
Then $r_x\in W'$, and hence $r_{t_n x}=t_n r_x t_n\in W'$, which in turn yields that $t_n x\in \Phi_1(W')$. 
Then it follows that 
$t_{n-1}t_n x\in \Phi_1(W')$ and so on. Thus 
$wx =t_1\cdots t_n x\in \Phi_1(W')$. Since $x\in \Phi_1(W')$ is arbitrary, 
it follows that $w\Phi_1(W')\subseteq \Phi_1(W')$. Finally, replacing 
$w \in W'$ by $w^{-1}$ we see that $\Phi_1(W') \subseteq w\Phi_1(W')$.
\end{proof}

\begin{remark}
Let $W'$ be a reflection subgroup. For each $i\in \{1, 2\}$, 
it follows from the above lemma that $\Phi_i(W')$ is a root
subsystem of $\Phi_i$, and we call it the \emph{root subsystem
corresponding to} $W'$. It is easily seen that there is a 
bijective correspondence between the set of reflection 
subgroups $W'$ of $W$ and the set of root subsystems $\Phi'_i$
of $\Phi_i$: $W'$ uniquely determines the corresponding root 
subsystem $\Phi_i(W')$; and $\Phi'_i$ uniquely determines the
reflection subgroup $W':=\langle \{\, r_x\mid x\in \Phi'_i\,\}\rangle$.
\end{remark}

In fact, for a reflection subgroup $W'$, we shall see that $\Phi_1(W')$ 
and $\Phi_2(W')$ are the root systems for the Coxeter group $W'$ arising from 
a suitably chosen Coxeter datum. In order to do this, we need a few preparatory 
results first.

\begin{remark}
For each $i\in \{1, 2\}$, it has been observed in \cite{FU2} that non-trivial scalar multiple
of an element of $\Phi_i$ can still be an element
of $\Phi_i$ (see the example immediately after \cite[Definition 3.1]{FU2} and 
\cite[Lemma 3.20]{FU2}). Therefore, unlike in the classical setting of \cite{HM}, 
we do not have a bijection from $T$ to either $\Phi^+_1$ or $\Phi^+_2$. 
\end{remark}

\begin{definition}
For each $i\in \{1, 2\}$, define an equivalence relation $\sim_i$ on $\Phi_i$ as 
follows: if $z_1, z_2\in \Phi_i$, then $z_1\sim_i z_2$ if and only if $z_1$ and
$z_2$ are (non-zero) scalar multiples of each other. For each $z\in \Phi_i$, write 
$\widehat{z}$ for the equivalence class containing $z$ and write 
$\widehat{\Phi_i}=\{\, \widehat{z}\mid z\in \Phi_i\,\}$.
\end{definition}

\begin{remark}
Observe that $W$ has a natural action on $\widehat{\Phi_i}$ (for each $i\in \{1, 2\}$) 
given by $w\widehat{z}=\widehat{wz}$ for all $w\in W$ and $z\in \Phi_i$. Furthermore, 
given $z, z'\in \Phi_i$, the corresponding reflections $r_z$ and $r_{z'}$ are equal if 
and only if $\widehat{z}=\widehat{z'}$.
\end{remark}

\begin{definition}
\label{def:N}
For $i\in \{1, 2\}$, and for each $w\in W$, define 
$$
N_i(w) = \{\, \widehat{z} \mid \text{$z \in \Phi^+_i$ and
$w z\in \Phi^-_i$}\, \}. 
$$
\end{definition}
Note that for $w\in W$, the set $N_i(w)$ ($i=1, 2$) can be
alternatively characterized as 
$
\{\, \widehat{z} \mid \text{$z \in \Phi^-_i$ and
$w z\in \Phi^+_i$}\, \} 
$.
Hence $\widehat{z}\in N_i(w)$ if and only if precisely one element of the
set
$\{z, wz\}$ is in $\Phi^+_i$.

\begin{notation}
\textup{Let $\ell: W\to \N$ denote the \emph{length function} with respect to
$(W, R)$, that is, for $w\in W$, 
$$\ell(w)=\min\{\, n\in \N\mid \text{ $w=r_1 r_2\cdots r_n$, where $r_1, r_2, \cdots, r_n \in R $  }\,\}.$$}
\end{notation}

A mild generalization of the techniques used in
(\cite[\S 5.6]{HM}) then yields the following connection between the length 
function and the functions $N_1$ and $N_2$:     

\begin{lemma}\textup{(\cite[Lemma 3.8]{FU2})}
\label{lem:simp}
{\rm (i)} $N_1(r_s)=\{\widehat{\alpha_s}\}$ and
$N_2(r_s)=\{\widehat{\beta_s}\}$ for all $s\in S$.
 
\noindent\rlap{\rm (ii)}\qquad Let $w\in W\!$. Then $N_1(w)$ and $N_2(w)$ both
have cardinality~$\ell(w)$.

\noindent\rlap{\rm (iii)}\qquad Let $w_1$, $w_2\in W$ and let $\dotplus $ denote
set symmetric difference. Then 
$
N_i(w_1w_2)=w_2^{-1}N_i(w_1)\dotplus N_i(w_2) 
$ for each $i\in \{1, 2\}$.
\qed
\end{lemma}

The last lemma enables us to deduce the following generalization of \cite[Lemma 3.2 (ii)]{FU2}:
\begin{proposition}
\label{pp:key}
For each $i\in \{1, 2\}$, let $w\in W$ and $x\in \Phi_i^+$. If $\ell(wr_x)>\ell(w)$ 
then $wx\in \Phi_i^+$, whereas if $\ell(wr_x)<\ell(w)$ then $wx\in \Phi_i^-$.
\end{proposition}
\begin{proof}
We prove the statement that $\ell(wr_x)>\ell(w)$ if and only if $wx$ is positive in the case 
$x\in \Phi_1$, and again we stress that a similar argument also 
shows the desired result holds in $\Phi_2$.

Observe that the second statement follows from the first, applied to $w r_x$ in place 
of $w$: indeed if $\ell(wr_x)<\ell(w)$ then $\ell((wr_x)r_x)>\ell(wr_x)$, forcing 
$(wr_x)x=w(r_x x)=-wx\in \Phi_1^+$, that is, $wx\in \Phi_1^-$.

Now we prove the first statement in $\Phi_1$. Proceed by induction on 
$\ell(w)$, the case $\ell(w)=0$ being trivial. If $\ell(w)>0$, then there exists $s\in S$ 
with $\ell(r_sw)=\ell(w)-1$, and hence
$$
\ell((r_s w)r_x)=\ell(r_s(wr_x))\geq \ell(w r_x)-1>\ell(w)-1=\ell(r_s w).
$$
Then the inductive hypothesis yields that $(r_sw)x\in \Phi_1^+$. 
Suppose for a contradiction that $wx\in \Phi_1^-$. Then $\widehat{wx}\in N_1(r_s)$ 
and Lemma~\ref{lem:simp}~(i) yields that $wx=-\lambda \alpha_s$ for some $\lambda>0$. 
But then $r_s w x =\lambda \alpha_s$, and hence $(r_s w)r_x(r_s w)^{-1}=r_s$ 
by calculations similar to (\ref{eq:ref3}) and (\ref{eq:ref4}). But this yields that $w r_x =r_s w$, 
contradicting $\ell(wr_x)>\ell(w)>\ell(r_s w)$, as desired.
\end{proof} 

\begin{definition}
 For each $w\in W$, define 
$$\overline{N}(w):=\{\, t\in T\mid \ell(wt)<\ell(w)\,\}.$$ 
\end{definition}

If $t\in T$ then $t=w r_s w^{-1}$ for some $w\in W$ and $s\in S$, and hence it follows from 
calculations like (\ref{eq:ref3}) and (\ref{eq:ref4}) that $t=r_{w\alpha_s}=r_{w\beta_s}$. 
This combined with Proposition~\ref{pp:key} give us:

\begin{proposition}
 \label{pp:equal}
Let $w\in W$. Then 
$$\overline{N}(w)=\{\, r_x\mid \widehat{x}\in N_i(w)\,\}$$
for each $i\in \{1, 2\}$.
\qed
\end{proposition}

\begin{definition}
Suppose that $W'$ is a reflection subgroup. Then we define
\begin{align*}
S(W')&:=\{\, t\in T\mid \overline{N}(t)\cap W'=\{t\}\,\}\\
\noalign{\hbox{and}}
\Delta_i(W')&:=\{x\in \Phi^+_i\mid r_x\in S(W')\,\}
\end{align*}
for each $i\in \{1, 2\}$.
\end{definition}

For a reflection subgroup $W'$, the set $S(W')$ is called the \emph{canonical generators} of $W'$ in \cite{MD87}, 
and it is well known that $(W', S(W'))$ is a Coxeter system. Indeed, we have:
\begin{lemma} \textup{\cite{MD87}}
\label{lem:dyer}
Let $W'$ be a reflection subgroup of $W$.

\noindent\rm{(i)}\quad If $t\in W' \cap T$, 
then there exist $m\in \N$ and $t_0, \cdots, t_m \in S(W')$ 
such that $t= t_m \cdots t_1 t_0 t_1 \cdots t_m$.

\noindent\rm{(ii)}\quad  $(W', S(W'))$ is a Coxeter system.
\end{lemma}
\begin{proof}
 (i):\quad \cite[Lemma (1.7) (ii)]{MD87}.

 \noindent\rm{(ii)]}:\quad \textup{\cite[Theorem (1.8) (i)]{MD87}}.
\end{proof}

For a reflection subgroup $W'$, we will show that 
$\Delta_1(W')$ and $\Delta_2(W')$ can be characterized in terms of a suitably defined Coxeter datum. 
Before we could prove this, we need a number of simple observations.

Observe that for a reflection subgroup $W'$ we can equivalently 
define $\Delta_i(W')$ by requiring 
\begin{equation}
\label{eq:alt}
 \Delta_i(W'):=\{\,x\in \Phi^+_i\mid N_i(r_x)\cap \widehat{\Phi_i(W')}=\{\widehat{x}\}\,\}.
\end{equation}

Suppose that $\Delta'_1\subseteq \Phi_1^+$ and $\Delta'_2\subseteq \Phi_2^+$ are two sets of roots
with $\phi(\Delta'_1)=\Delta'_2$ (where $\phi$ is as in Lemma~\ref{lem:equivariant}~(iii)). Furthermore, 
suppose that  $\Delta'_1$ and  $\Delta'_2$ satisfy the following:
\begin{itemize}
 \item [(i')] $\langle x, \phi(x')\rangle \leq 0$, for all distinct $x, x'\in \Delta'_1$;
 \item [(ii')] $\langle x, \phi(x')\rangle \langle x', \phi(x)\rangle \in 
               \{\, \cos^2(\pi/m) \mid m\in \N, m\geq 2\,\}\cup [1, \infty)$, for all
              $x, x'\in \Delta'_1$ with $r_x\neq r_{x'}$.
\end{itemize}
It follows from Lemma \ref{lem:equivariant} that
\begin{equation}
 \label{eq:D1}
\langle x, \phi(x)\rangle =1, \text{ for all
$x\in \Delta'_1$}.
\end{equation}
Since $\Delta'_1\subseteq \PLC(\Pi_1)$ and $\Delta'_2\subseteq \PLC(\Pi_2)$, it follows 
that $0\notin \PLC(\Delta'_1)$ and $0\notin \PLC(\Delta'_2)$. Also it can be readily 
checked from (i'), (ii') and (\ref{eq:D1}) that $x\notin \PLC(\Delta'_1\setminus\{x\})$ and
$\phi(x)\notin \PLC(\Delta'_2\setminus\{\phi(x)\})$ for all $x\in \Delta'_1$. 
Furthermore, Lemma~\ref{lem:equivariant}~(iii) ensures that $\langle x, \phi(x')\rangle =0$
whenever $\langle x', \phi(x)\rangle =0$ for all $x, x'\in\Delta'_1\subseteq \Phi_1$, 
Thus $\Delta'_1$ and $\Delta'_2$ satisfy conditions (D1) to (D5) of the present paper inclusive. If we let 
$S'$ be an indexing set for both $\Delta'_1$ and $\Delta'_2$ then
$$\mathscr{C'}:=(\,S', \spa(\Delta'_1), \spa(\Delta'_2), \Delta'_1, \Delta'_2, \langle\,,\,\rangle'     \,),$$
(where $\langle\,,\,\rangle'$ denotes the restriction of $\langle\,,\,\rangle$ to 
$\spa(\Delta'_1)\times \spa(\Delta'_2)$) constitutes a Coxeter datum in the sense of \cite{FU2}.
Now if we let $R':=\{\, r_x\mid x\in \Delta'_1\,\} (=\{\, r_y\mid y\in \Delta'_2\,\})$, and set $W'=\langle R'\rangle$, 
then it is clear that $W'$ is a reflection subgroup of $W$. Furthermore, it follows from 
Theorem~\ref{th:key} that $(W', R')$ forms a Coxeter system. Then upon applying Lemma~\ref{lem:simp} and (\ref{eq:alt}) to 
$\mathscr{C'}$ and $W'$ we may conclude that $S(W')=R'$ and consequently $\widehat{\Delta_1(W')}=\widehat{\Delta'_1}$
and $\widehat{\Delta_2(W')}=\widehat{\Delta'_2}$. Summing up, we have:

\begin{proposition}
 \label{pp:can1}
Suppose that $\Delta'_1\subseteq \Phi_1^+$ and $\Delta'_2\subseteq \Phi_2^+$ such that

 \noindent \rm{(A1)}\quad $\phi(\Delta'_1)=\Delta'_2$;
 
 \noindent \rm{(A2)}\quad $\langle x, \phi(x')\rangle \leq 0$, for all distinct $x, x'\in \Delta'_1$;
  
 \noindent \rm{(A3)}\quad $\langle x, \phi(x')\rangle \langle x', \phi(x)\rangle \in 
               \{\, \cos^2(\pi/m) \mid m\in \N, m\geq 2\,\}\cup [1, \infty)$, for all
              $x, x'\in \Delta'_1$ with $r_x\neq r_{x'}$.

\noindent Then $W'=\langle\{\, r_x\mid x\in \Delta'_1\,\}\rangle$ is a reflection subgroup of $W$ with
$\widehat{\Delta'_1}=\widehat{\Delta_1(W')}$ and $\widehat{\Delta'_2}=\widehat{\Delta_2(W')}$.
\qed
\end{proposition}

It turns out that the converse of Proposition \ref{pp:can1} is also true, namely: if $W'$ is a reflection 
subgroup of $W$, and if $x, x'\in \Delta_1(W')$ with $r_x \neq r_{x'}$, then 
conditions (A2) and (A3) of Proposition~\ref{pp:can1} must be satisfied. Since 
Lemma~\ref{lem:equivariant}~(iii) ensures that $\langle x, \phi(x')\rangle =0$
if and only if $\langle x' \phi(x)\rangle =0$, it follows from this assertion and 
a quick argument similar to the one used immediately after (\ref{eq:D1}) that 
representative elements from $\Delta_1(W')$ and $\Delta_2(W')$ can be used to form a Coxeter datum for $W'$.
Hence this assertion and Proposition~\ref{pp:can1} together yield
that for a reflection subgroup $W'$, the corresponding $\Delta_i(W')$ ($i=1, 2$) can be 
characterized by a suitable Coxeter datum.
We devote the rest of this section to a proof of this assertion.

\begin{lemma}
\label{lem:dfacts}
Let $W'$ be a reflection subgroup of $W$.

\noindent{(i)} For each $i\in\{1,2\}$, let $x\in \Pi_i\setminus \Phi_i(W')$. Then
$\Delta_i(r_x W' r_x)=r_x \Delta_i(W')$.

\noindent{(ii)} For each $i\in \{1,2\}$, $\Phi_i(W')=W'\Delta_i(W')$.
\end{lemma}
\begin{proof}
(i):\quad It is readily checked that $r\Phi_i(W')=\Phi_i(rW'r)$ 
for all $r\in T$. Since $x\in \Pi_i\setminus \Phi_i(W')$, it follows 
that $r_x \in R\setminus W'$. Let $y\in \Delta_i(W')$ be arbitrary. Then 
\begin{align*}
 \qquad N_i(r_{(r_xy)})\cap \widehat{\Phi_i(r_x W' r_x)}
&=N_i(r_x r_y r_x)\cap \widehat{\Phi_i(r_x W' r_x)}\\
&\qquad\qquad\qquad \text{(by (\ref{eq:ref3}) and (\ref{eq:ref4}))}\\
&=(r_x N_i(r_x r_y)\dotplus N_i(r_x))\cap \widehat{\Phi_i(r_x W' r_x)} \\
&\qquad\qquad\qquad\text{(by Lemma \ref{lem:simp} (iii))}\\
&=(r_x r_y N_i(r_x)\dotplus r_x N_i(r_y)\dotplus N_i(r_x))\\
&\qquad\qquad\cap \widehat{\Phi_i(r_x W' r_x)}\\
&\qquad\qquad\qquad\text{(again by Lemma \ref{lem:simp} (iii))}\\
&=r_x ((r_y N_i(r_x)\dotplus N_i(r_y)\dotplus N_i(r_x))\\
&\qquad\qquad\cap \widehat{\Phi_i( W' )})\\
&=r_x((r_y \{\widehat{x}\}\dotplus N_i(r_y)\dotplus \{\widehat{x}\})\cap \widehat{\Phi_i( W' )})\\
&\qquad\qquad\qquad\text{(by Lemma \ref{lem:simp} (i))}\\
&=r_x(N_i(r_y)\cap \widehat{\Phi_i( W' )})\\
&\qquad\qquad\qquad\text{(since $\widehat(x), r_y\widehat{x}\notin \widehat{\Phi_i(W')}$)}\\
&=\{\widehat{r_xy}\}\\
&\qquad\qquad\qquad\text{(since $y\in \Delta_i(W')$).}
\end{align*}
Hence $r_x y\in \Delta_i(r_x W' r_x)$. This proves that 
$r_x \Delta_i(W')\subseteq \Delta_i(r_x W' r_x)$. 
But $x\in \Pi_i\setminus r_x\Phi_i(W')$, so the above yields that
 $ r_x \Delta_i(r_xW' r_x)\subseteq \Delta_i(W')$ proving the desired result.

(ii):\quad Since $\Delta_i(W') \subseteq \Phi_i(W')$ for each
 $i\in \{1,2\}$, it follows from Lemma~\ref{lem:stab} that 
$W'\Delta_i(W')\subseteq \Phi_i(W')$.

Conversely if $x\in \Phi_i(W')$ then $r_x \in W'\cap T$. By Lemma~\ref{lem:dyer}~(i), 
there are $x_0, x_1, \cdots, x_m \in \Delta_i(W')$ ($m\in \N$) such that 
$$r_x = r_{x_m}\cdots  r_{x_1} r_{x_0} r_{x_1}\cdots r_{x_m}.$$
Calculations similar to those of (\ref{eq:ref3}) and (\ref{eq:ref4}) 
enable us to conclude that 
$\lambda x =r_{x_m}\cdots r_{x_1}x_0 \in W' \Phi_i(W')$
for some (nonzero) scalar 
$\lambda$. Now since $\tfrac{1}{\lambda} x_0 =(r_{x_m}\cdots r_{x_1})^{-1} x\in \Phi_i$, 
it follows that $\tfrac{1}{\lambda}x_0 \in \Delta_i(W')$ and hence 
$x=r_{x_m}\cdots r_{x_1}(\tfrac{1}{\lambda}x_0)\in W'\Delta_i(W')$ as required.
\end{proof}
 
\begin{definition}
Let $W'$ be a reflection subgroup of $W$, and let $\ell_{W'}:W'\to \N$ 
be the length function on $(W', S(W'))$ defined by 
$$\ell_{W'}(w)=\min\{\,n\in \N \mid \text{$w=r_1\cdots r_n$, where $ r_i \in S(W')$}\,\}.$$
If $w=r_1\cdots r_n \in W'$ ($r_i \in S(W')$) and $n=\ell_{W'}(w)$ then $r_1\cdots r_n$ is 
called a \emph{reduced expression} for $w$ (with respect to $S(W')$).
\end{definition}

\begin{lemma}
\label{lem:N_i'}
Let $W'$ be a reflection subgroup. For each $i\in \{1,2\}$, 
\begin{itemize}
\item [(i)] $N_i(r_x)\cap \widehat{\Phi_i(W')}=\{\widehat{x}\}$ for all $x\in \Delta_i(W')$;
\item [(ii)] for all $w_1\in W$ and $w_2\in W'$
$$N_i(w_1w_2) \cap \widehat{\Phi_i(W')}= w_2^{-1}(N_i(w_1)\cap \widehat{\Phi_i(W')})\dotplus (N_i(w_2)\cap \widehat{\Phi_i(W')}).$$
\end{itemize}
\end{lemma}
\begin{proof}
(i)\quad is just the definition of $\Delta_i(W')$.

(ii)\quad Lemma \ref{lem:simp} (iii) yields that $N_i(w_1w_2) = w_2^{-1}N_i(w_1) \dotplus  N_i(w_2)$, and hence 
$$N_i(w_1w_2)\cap \widehat{\Phi_i(W')}=(w_2^{-1}N_i(w_1)\cap \widehat{\Phi_i(W')})\dotplus (N_i(w_2)\cap \widehat{\Phi_i(W')}).$$
Since $w_2\in W'$ it follows from Lemma \ref{lem:stab} that $w_2^{-1}\widehat{\Phi_i(W')}=\widehat{\Phi_i(W')}$.
Thus $w_2^{-1}N_i(w_1)\cap \widehat{\Phi_i(W')}= w_2^{-1}(N_i(w_1)\cap \widehat{\Phi_i(W')})$, giving us
$$
N_i(w_1w_2) \cap \widehat{\Phi_i(W')}= w_2^{-1}(N_i(w_1)\cap \widehat{\Phi_i(W')})\dotplus (N_i(w_2)\cap \widehat{\Phi_i(W')}).
$$
\end{proof}

\begin{lemma}
\label{lem:length}
Let $W'$ be a reflection subgroup. For each $i\in \{1,2\}$ and all $w\in W'$, we have

\noindent\rlap{\rm (i)}\qquad $|N_i(w)\cap \widehat{\Phi_i(W')}|= \ell_{W'}(w)$. 
Furthermore, if $w=r_{x_1}\cdots r_{x_n}$ (where $x_1, \cdots, x_n\in \Delta_i(W')$)
is reduced with respect to $(W', S(W'))$ then 
$$N_i(w)\cap \widehat{\Phi_i(W')} =\{\widehat{y_1}, \cdots \widehat{y_n}\}$$ 
where $y_j= (r_{x_n}\cdots r_{x_{j+1}})x_j$ for all $j=1,\cdots, n$.

\noindent\rlap{\rm (ii)}\qquad 
$N_i(w)\cap \widehat{\Phi_i(W')} =\{\widehat{x}\in \widehat{\Phi_i(W')} \mid \ell_{W'}(wr_x)<\ell_{W'}(w)\}$.
\end{lemma}

\begin{proof}
(i): For each $j\in \{1, \cdots, n\}$, set $t_j = r_{x_n}\cdots r_{x_{j+1}}r_{x_j}r_{x_{j+1}}\cdots r_{x_n}$,
 that is, $t_j=r_{y_j}$. If $t_j=t_k$ where $j>k$ then 
\begin{align*}
w&= r_{x_1}\cdots r_{x_{k-1}}r_{x_{k+1}}\cdots r_{x_n}t_k \\
&=r_{x_1}\cdots r_{x_{k-1}}r_{x_{k+1}}\cdots r_{x_n}t_j\\
&=r_{x_1}\cdots r_{x_{k-1}}r_{x_{k+1}}\cdots r_{x_{j-1}}r_{x_{j+1}}\cdots r_{x_n}
\end{align*}
contradicting $\ell_{W'}(w)=n$. Hence the $t_j$'s are all distinct and consequently all the $\widehat{y_j}$'s 
are all distinct. Now by repeated application of Lemma~\ref{lem:N_i'}~(ii), for each $i\in \{1,2\}$ we have
\begin{align*}
&\qquad N_i(w)\cap \widehat{\Phi_i(W')}\\
&= (N_i(r_{x_n}\cap \widehat{\Phi_i(W')})\dotplus r_{x_n}(N_i(r_{n-1})\cap\widehat{\Phi_i(W')})\dotplus \cdots \\
&\qquad \qquad \qquad \qquad \qquad\qquad\qquad\qquad\dotplus r_{x_n}\cdots r_{x_2}(N_i(r_{x_1})\cap \widehat{\Phi_i(W')})\\
&= \{\widehat{y_n}\} \dotplus  \{\widehat{y_{n-1}}\} \dotplus \cdots \dotplus \{\widehat{y_1}\} \\
&=\{\,\widehat{y_1},\cdots, \widehat{y_n} \,\}
\end{align*}
and consequently $|N_i(w) \cap \widehat{\Phi_i(W')}| =\ell_{W'}(w)$.

(ii): Let $w=r_{x_1}\cdots r_{x_n}$ be a reduced expression for $w\in W'$ with respect to 
$S(W')$ ($x_1, \cdots, x_n \in \Delta_i(W')$). Then for each $i\in \{1,2\}$, Part~(i) above yields that 
$$N_i(w)\cap \widehat{\Phi_i(W')}= \{\,\widehat{y_1},\cdots, \widehat{y_n}\,\}$$
where  $y_j= (r_{x_n}\cdots r_{x_{j+1}})x_j$, for all $j\in \{1, \cdots, n\}$. Now for each such~$j$,
$$
wr_{y_j}= w r_{x_n}\cdots r_{x_{j+1}}r_{x_j} r_{x_{j+1}} \cdots r_{x_n}=r_{x_1}\cdots r_{x_{j-1}}r_{x_{j+1}}\cdots r_{x_n}
$$
and so $\ell_{W'}(wr_{y_j})\leq n-1 <\ell_{W'}(w)$. Hence if $\widehat{x} \in N_i(w)\cap \widehat{\Phi_i(W')}$, 
then $\ell_{W'}(w r_x)< \ell_{W'}(w)$.

Conversely, suppose that $x\in \Phi_i(W') \cap \Phi^+_i$ and $\widehat{x}\notin N_i(w)$. 
We are done if we could show that then $\ell_{W'}(w r_x)>\ell_{W'}(w)$. Observe that the given choice of $x$ implies 
that $\widehat{x}\in N_i(r_x) \cap \widehat{\Phi_i(W')}$, furthermore, 
$\widehat{x}\notin r_x (N_i(w)\cap \widehat{\Phi_i(W')})$. Therefore 
$$\widehat{x} \in r_x(N_i(w)\cap \widehat{\Phi_i(W')})\dotplus  (N_i(r_x)\cap \widehat{\Phi_i(W')}) = N_i(wr_x)\cap \widehat{\Phi_i(W')},$$
and by what has just been proved, this implies that
$$\ell_{W'}(w)= \ell_{W'}((w r_x)r_x)<\ell_{W'}(w r_x),$$
as desired.

\end{proof}

The following is a mild generalization of \cite[Lemma 3.2]{MD87}:
\begin{lemma}
\label{lem:d3.2}
Let $W'$ be a reflection subgroup. For each $i\in \{1,2\}$, let $x,y \in \Delta_i(W')$ 
such that $r_x \neq r_y$. Let $n =\ord(r_x r_y)$. Then for $0\leq m<n$
$$
\underbrace{\cdots r_y r_x r_y}_\text{$m$ factors} x \in \Phi_i^+ \qquad \text{and}\qquad
\underbrace{\cdots r_x r_y r_x}_\text{$m$ factors} y \in \Phi_i^+.$$
\end{lemma}
\begin{proof}
It is easily checked that when $0\leq m<n$ we have
$$\ell_{W'}((\underbrace{\cdots r_y r_x r_y}_\text{$m$ factors})r_x) = m+1 >m=\ell_{W'}(\underbrace{\cdots r_y r_x r_y}_\text{$m$ factors}),$$ 
as well as
$$\ell_{W'}((\underbrace{\cdots r_x r_y r_x}_\text{$m$ factors})r_y) = m+1 >m=\ell_{W'}(\underbrace{\cdots r_x r_y r_x}_\text{$m$ factors}).$$
Hence the desired result follows immediately from Lemma~\ref{lem:length}.
\end{proof}

In fact we can refine Lemma \ref{lem:d3.2} with the following generalization of \cite[Lemma 3.3]{MD87}:
\begin{lemma} 
\label{lem:d3.3}
Let $W'$ be a reflection subgroup. For each $i\in \{1,2\}$, 
let $x, y \in \Delta_i(W')$ with $r_x \neq r_y$. Let $n=\ord(r_x r_y)$, and let 
$c_m, d_m, c'_m$ and $d'_m$ be constants such that
$$\underbrace{(\cdots r_y r_x r_y)}_\text{$m$ factors} x = c_m x + d_m y \quad \text{and} \quad \underbrace{(\cdots r_x r_y r_x)}_\text{$m$ factors} y = c_m' x + d_m' y.$$
Then $c_m \geq 0$, $d_m \geq 0$, $c_m' \geq 0$ and $d_m'\geq 0$ whenever $m<n$.
\end{lemma}
\begin{proof}
By symmetry, it will suffice to prove that $d_m\geq 0$ and $d_m'\geq 0$. The proof of this will be based on an induction on $\ell(r_x)$.

Suppose first that $\ell(r_x)=1$. Then $\lambda x\in \Pi_i$ for some $\lambda>0$. 
Write $y=\sum\limits_{z\in \Pi_i}\lambda_z z$ where $\lambda_z\geq 0$ for all $z\in \Pi_i$. 
In fact, $\lambda_{z_0} >0$ for some $z_0\in \Pi_i\setminus\{x\}$, since otherwise we would have $y\in \R x$ and so $r_x=r_y$. 
Now for $0\leq m<n$, Lemma \ref{lem:d3.2} yields that 
$$
(\underbrace{\cdots r_y r_x r_y}_\text{$m$ factors})x =c_m x+ \sum\limits_{z\in \Pi_i} d_m\lambda_z z\in \Phi_i^+. 
$$
That is
$$c_m x+d_m (\sum_{z\in \Pi_i}\lambda_z z)=\sum_{z\in \Pi_i}\mu_{z}z, \text{ where $\mu_{z}\geq 0$, for all $z\in \Pi_i$}.$$
Now if $d_m\leq 0$ then the above yields that
$$(c_m-\mu_x) x=(\mu_{z_0}-d_m \lambda_{z_0})z_0+ \sum_{z\in \Pi_i\setminus\{x, z_0\}}(\mu_{z}-d_m\lambda_z )z, $$
contradicting condition (D2). Therefore, $d_m>0$ as required.
Similarly $d_m'\geq 0$.

Suppose inductively now that the result is true for reflection subgroups $W''$ of $W$ 
and $x',\,y'\in \Delta_i(W'')$ with $r_{x'}\neq r_{y'}$ and $\ell(r_{x'})<\ell(r_x)$ where $\ell(r_x)\geq 3$. 
It is well know that there exists $z\in \Pi_i$ such that $\ell(r_z r_x r_z)=\ell(r_x)-2$. 
Then $\ell(r_x r_z)<\ell(r_x)$, and thus $\widehat{z}\in N_i(r_x)$. But since $x\in \Delta_i(W')$ 
and $x\neq z$ (since $\ell(r_x)\geq 3$), it follows that $r_z \notin W'$. Let $W''=r_z W' r_z$. 
Lemma~\ref{lem:dfacts}~(i) yields that $\Delta_i(W'')=r_z \Delta_i(W')$ and therefore $r_z x,\,r_z y \in \Delta_i(W'')$. Now 
\begin{equation}
\label{eq:r_z}
r_{(r_z x)}= r_z r_x r_z \qquad\text{and}\qquad r_{(r_z y)}=r_z r_y r_z
\end{equation}
and hence $\ord(r_{(r_z x)} r_{(r_z y)})=\ord(r_x r_y)=n$. Since $\ell(r_{(r_z x)})=\ell(r_x)-2$, the inductive hypothesis gives
\begin{align*}
(\underbrace{\cdots r_{(r_z y)} r_{(r_z x)} r_{(r_z y)}}_\text{$m$ factors})(r_z x)&= c_m(r_z x)+d_m (r_z y)\\
\noalign{\hbox{and}}
(\underbrace{\cdots r_{(r_z x)} r_{(r_z y)} r_{(r_z x)}}_\text{$m$ factors})(r_z y) &=c_m'(r_z x)+d_m'(r_z y)
\end{align*}
where $d_m, \,d_m'\geq 0$ for $0\leq m<n$. Finally, by (\ref{eq:r_z}), the desired result follows 
on applying $r_z$ to both sides of the last two equations.
\end{proof}

\begin{proposition}
\label{pp:d3.4}
Let $W'$ be a reflection subgroup of $W$. Suppose that $x,y \in \Delta_1(W')$ with $r_x \neq r_y$. 
Let $n=\ord(r_x r_y) \in \{\infty\} \cup \N$. Then
\begin{align*}
&\langle x, \phi(y)\rangle \leq 0 \\
\noalign{\hbox{and}}
&\begin{cases}
\langle x, \phi(y)\rangle \langle y, \phi(x)\rangle = \cos^2\tfrac{\pi}{n}\qquad & (n\in \N, n \geq 2)\\
\langle x, \phi(y)\rangle \langle y, \phi(x)\rangle \in [ 1, \infty) \qquad&(n=\infty)
\end{cases}
\end{align*}
\end{proposition}
\begin{proof}
Observe that since $r_{\phi(x)}=r_x\neq r_y=r_{\phi(y)}$, it follows that $\{x, y\}$ and 
$\{\phi(x), \phi(y)\}$ are both linearly independent, and hence conditions (D1) and (D2) are satisfied. 
Now let us set
$R_1'' = R_2'' :=\{r_x, r_y\}$ and $W_1'' =W_2'': =\langle \{r_x, r_y \}\rangle$, 
and furthermore, $ \Phi_1'':= W_1'' \{x, y\}$.
Observe that $\Phi''_1$ consists of elements of the form $\pm(\underbrace{\cdots r_y r_x r_y}_\text{$m$ factors})x$ 
and $\pm(\underbrace{\cdots r_x r_y r_x}_\text{$m$ factors})y$ (where $0\leq m <\ord(r_x r_y)$). Lemma~\ref{lem:d3.3} then yields that 
$\Phi''_1=\Phi''^+_1 \uplus \Phi''^-_1$, and consequently
Theorem \ref{th:datum} yields that
$$
\begin{cases}
\langle x, \phi(y)\rangle \langle y, \phi(x)\rangle = \cos^2\tfrac{\pi}{n}\qquad & (n\in \N, n \geq 2)\\
\langle x, \phi(y)\rangle \langle y, \phi(x)\rangle \in [ 1, \infty) &(n=\infty).
\end{cases}
$$
\end{proof}

\section{Acknowledgments}
A few results presented in this paper are taken from the author's PhD
thesis~\cite{FU0} and the author wishes to thank A/Prof.~R.~B.~Howlett for all
his help and encouragement throughout the author's PhD candidature.  
Due gratitude must be paid to Prof.~M.~Dyer for his penetrating insight and valuable suggestions.
The author also wishes to thank Prof.~G.~I.~Lehrer and Prof.~R.~Zhang for supporting this work.

\bibliographystyle{amsplain}

\begin{thebibliography}{4}
\bibitem{ABFB}
A. ~Bj\"{o}rner and F. ~Brenti, \emph{Combinatorics of Coxeter Groups}, Graduate
Texts in Mathematics, GTM 231, Springer, 2005 

\bibitem{BN68}
N.~Bourbaki, \emph{Groupes et algebras de Lie, Chapitres 4, 5 et 6 }, Hermann,
Paris, 1968


\bibitem{BH93}
B.~Brink and R.~B.~Howlett, ``A finiteness property and an automatic
structure of Coxeter groups'', \emph{Math. Ann.}, \textbf{296} (1993), 179--190.

\bibitem{BB98}
B.~Brink, ``The set of dominance-minimal roots'', \emph{J. Algebra}, \textbf{206}
(1998), 371--412.

\bibitem{BBHT}
F.~Bergeron, N.~Bergeron, R.~B. ~Howlett and D.~E.~Taylor, ``A Decomposition
of the Descent Algebra of a Finite Coxeter Group'', \emph{Journal of Algebraic
Combinatorics}, \textbf{1} (1992), 23--44

\bibitem{PR09}
P.~E.~Caprace and B.~R\'emy, ``Groups with a root group datum'', \emph{Innov. 
Incidence Geom.}, \textbf{9} (2009), 5--77.

\bibitem{BC01}
W.~A.~Casselman, `Computation in Coxeter groups I. Multiplication', 
\emph{Electron. J. Combin.}, \textbf{9} (2002), no.1, Research Paper 25, 22 pp. (electronic). 

\bibitem{BC06}
W.~A.Casselman, `Computation in Coxeter groups II. Constructing Minimal Roots', 
\emph{Represent. Theory}, \textbf{12}(2008), 260--293.

\bibitem{VD82}
V.~Deodhar, ``On the root system of a Coxeter group'', \emph{Comm. Algebra},
\textbf{10} (1982), no. 6, 611--630.

\bibitem{MD87}
M.~Dyer, \emph{Hecke algebras and reflections in Coxeter groups}, PhD thesis, 
University of Sydney, 1987.

\bibitem{MD90}
M.~Dyer, ``Reflection Subgroups of Coxeter Systems'', \emph{J. Algebra},
\textbf{135} (1990), 57--73.

\bibitem{FU0}
X.~Fu, \emph{Root systems and reflection representations of Coxeter groups},
PhD thesis, University of Sydney, 2010.

\bibitem{FU1}
X.~Fu, ``The dominance hierarchy in root systems of Coxeter groups'',
\emph{J. Algebra}, \textbf{366} (2012), 187--204.

\bibitem{FU2}
X.\!~Fu,\,\emph{Non-orthogonal geometric realizations of Coxeter groups}, arXiv:1112.3429 [math.RT], 
preprint, 2011.

\bibitem{hee91}
J.-Y. ~H\'ee, ``Syst\`eme de racines sur un anneau commutatif totalement ordonn\'e'',  
\emph{Geom. Dedicata}, \textbf{37} (1991), 65--102. 

\bibitem{HH81}
H.~Hiller, \emph{Geometry of Coxeter Groups}. Research Notes in Mathematics 
54, Pitman (Advanced Publishing Program), Boston-London, 1981.

\bibitem{CH11}
C.~Hohlweg, J.P.~Labb\'e and V.~Ripoll, \emph{Asymptotical behaviour of roots of infinite Coxeter groups I}, 
arXiv:1112.5415 [math.GR], preprint, 2011.

\bibitem{RB80}
R.~B.~Howlett, ``Normalizers of parabolic subgroups of reflection groups'', 
\emph{J. London Math. Soc. (2)}, \textbf{21} (1980), no. 1, 62--80.

\bibitem{RB96} 
R.~B.~Howlett, \emph{Introduction to Coxeter groups}, Lectures given at ANU,
1996 (available at http://www.maths.usyd.edu.au/res/Algebra/How/1997-6.html). 

\bibitem{HT97}
R.~B.~Howlett, P.~J.~ Rowley and D.~E.~Taylor, ``On Outer Automorphism groups
of Coxeter Groups'', \emph{Manuscripta Math.}, \textbf{93} (1997), 499--513. 

\bibitem{HM}
J.~Humphreys, \emph{Reflection Groups and Coxeter Groups}, Cambridge Stud. 
Adv. Math. vol. 29, Cambridge Univ. Press, 1990.

\bibitem{VK}
V.~G.~Kac, \emph{Infinite-dimensional Lie algebras}, third edition, Cambridge
University Press, Cambridge, 1990.

\bibitem{DK94}
D.~Krammer, \emph{The conjugacy problem for Coxeter groups}, PhD thesis,
Universiteit Utrecht, 1994.

\bibitem{MX82}
G.~A.~ Maxwell, ``Sphere packings and hyperbolic reflection groups'', \emph{J. Algebra}, 
\textbf{79} (1982), 78--97.

\bibitem{MP89}
R.~V.~Moody and A.~Pianzola, ``On infinite root systems'', 
\emph{Trans. Amer. Math. Soc.}, \textbf{315} (1989), 661--696.

\bibitem{VB71}
\`E.~B.~Vinberg, ``Discrete Linear Groups Generated by Reflections'', 
\emph{Izv. Akad. Nauk SSSR Ser. Mat.}, \textbf{35} (1971), 1072--1112.

\end{thebibliography}

\providecommand{\bysame}{\leavevmode\hbox to3em{\hrulefill}\thinspace}
\providecommand{\MR}{\relax\ifhmode\unskip\space\fi MR }
\providecommand{\MRhref}[2]{%
  \href{http://www.ams.org/mathscinet-getitem?mr=#1}{#2}
}
\providecommand{\href}[2]{#2}

\end{document}